\documentclass[12pt]{amsart}
\usepackage{amsmath,amssymb}
\usepackage{amsfonts}
\usepackage{amsthm}
\usepackage{latexsym}
\usepackage{graphicx}

\def\p{\partial}

\def\R{\mathbb{R}}

\def\vv<#1>{\langle#1\rangle}

\def\1{\mathbf{1}}

\def\XXint#1#2{\setbox0=\hbox{$#1{#2}{\int}$}{#2}\kern-.5\wd0 }

\def\XXint#1#2#3{{\setbox0=\hbox{$#1{#2#3}{\int}$}
     \vcenter{\hbox{$#2#3$}}\kern-.5\wd0}}

\def\vv<#1>{{\left\langle#1\right\rangle}}

\def\CD{{\rm CD}}
\def\Deg{{\rm Deg}}
\newtheorem{thm}{Theorem}[section]

\newtheorem{lem}{Lemma}[section]

\theoremstyle{definition}

\theoremstyle{remark}

\numberwithin{equation}{section}

\begin{document}
\title{A Lichnerowicz-type estimate for Steklov eigenvalues on graphs and its rigidity}

\author{Yongjie Shi$^1$}
\address{Department of Mathematics, Shantou University, Shantou, Guangdong, 515063, China}
\email{yjshi@stu.edu.cn}
\author{Chengjie Yu$^2$}
\address{Department of Mathematics, Shantou University, Shantou, Guangdong, 515063, China}
\email{cjyu@stu.edu.cn}
\thanks{$^1$Research partially supported by NNSF of China with contract no. 11701355. }
\thanks{$^2$Research partially supported by NNSF of China with contract no. 11571215.}
\renewcommand{\subjclassname}{%
  \textup{2010} Mathematics Subject Classification}
\date{}
\begin{abstract}
In this paper, we obtain a Lichnerowicz-type estimate for the first Steklov eigenvalues on graphs and discuss its rigidity.
\end{abstract}
\maketitle\markboth{Shi \& Yu}{Lichnerowicz-type estimate for Steklov eigenvalues on graphs}
\section{Introduction}
Let $(M^n,g)$ be a closed Riemannian manifold with Ricci curvature not less than a positive constant $K$. Then, its first positive eigenvalue for the Laplacian operator must be no less than $\frac{nK}{n-1}$. This is a well-known result obtained by Lichnerowicz \cite{LI} which is now called the Lichnerowicz estimate. This result was later extended to the first Dirichlet eigenvalue for compact Riemannian manifolds with mean convex boundary by Reilly \cite{RE} via his famous integral formula which is now called the Reilly formula. For Steklov eigenvalues on compact Riemannian manifolds with boundary, there is a conjecture by Escobar
\cite{ES} in a similar spirit which was partially solved by Xia-Xiong \cite{XX} recently.

In the discrete setting, similar theory was developed. Two notions of Ricci curvature lower bounds for graphs, the Bakry-\'Emery curvature from the diffusion process point of view and the Ollivier curvature from the optimal transport point of view, were introduced and developed in \cite{BE,OL,LY,LLY,CLP, MW}. It is really interesting that classical results in the smooth setting such as Lichnerowicz estimate (\cite{OL,BCLL,KKRT,LLY,LMP2}), Bonnet-Myer diameter estimate (\cite{LLY,CK,LMP,LMP2}), Li-Yau gradient estimate (\cite{BH}) etc. all has its corresponding versions in the discrete setting.

Motivated by all these works, we consider Lichnerowicz-type estimates for Steklov eigenvalues on graphs in this paper. Let's recall some preliminaries before stating the main results of the paper.

Let $G$ be a graph with the set of vertices denoted by $V(G)$ and the set of edges denoted by $E(G)$. We will also write $V(G)$ and $E(G)$ as $V$ and $E$ for simplicity if no confusion was made. Without further indications, throughout this paper, graphs are assumed to be finite, simple and connected. For $A,B\subset V$, we will use $E(A,B)$ to denote the set of edges in $G$ joining a vertex in $A$ to a vertex in $B$.

A triple $(G,m,w)$ is called a weighted graph where $m$ and $w$ are positive functions on $V$ and $E$ which are called the vertex measure and edge weight respectively. For simplicity, the edge weight $w$ is also viewed as a symmetric function on $V\times V$ such that for any $x,y\in V$, $w(x,y)$ is the weight of the edge $\{x,y\}$ when $\{x,y\}\in E$, otherwise $w(x,y)=0$. We will also write $m(x)$ as $m_x$ and $w(x,y)$ as $w_{xy}$ for simplicity. For each $x\in V$, we call
\begin{equation}
\Deg(x):=\frac{1}{m_x}\sum_{y\in V}w_{xy}
\end{equation}
the weighted degree of $x$. Moreover, for any $S\subset V$, we call
\begin{equation}
V_S:=\sum_{x\in S}m_x
\end{equation}
the volume of $S$.

Some special weights are attracting special interests when studying the structure of graphs. For example, the unit weight with $m\equiv 1$ and $w\equiv 1$ is usually considered when considering combinatorial structure of graphs. The normalized weight with $\Deg(x)=1$ for any $x\in V$ is usually considered when considering random walks on graphs.

For $u\in \R^V$, the Laplacian of $u$ is defined as
\begin{equation}
\Delta u(x)=\frac{1}{m_x}\sum_{y\in V}(u(y)-u(x))w_{xy}
\end{equation}
for any $x\in V$. The differential of $u$ is a skew-symmetric function on $V\times V$ defined as:
\begin{equation}
du(x,y)=\left\{\begin{array}{ll}u(y)-u(x)&\{x,y\}\in E\\0&\mbox{otherwise.}\end{array}\right.
\end{equation}
In fact, as mentioned in \cite{SY}, skew-symmetric functions $\alpha$ on $V\times V$ such that $\alpha(x,y)=0$ whenever $x\not\sim y$ are called $1$-forms of $G$. The space of $1$-forms on $G$ is denoted as $A^1(G)$. A natural inner product on $A^1(G)$ is defined as:
\begin{equation}
\vv<\alpha,\beta>=\sum_{\{x,y\}\in E}\alpha(x,y)\beta(x,y)w_{xy}=\frac12\sum_{x,y\in V}\alpha(x,y)\beta(x,y)w_{xy}
\end{equation}
for any $\alpha,\beta\in A^1(G)$. Note that functions on $V$ can be viewed as $0$-forms. So the space $A^0(G)$ of $0$-forms on $G$ is just $\R^V$. The natural inner product on $A^0(G)$ is:
\begin{equation}
\vv<u,v>=\sum_{x\in V}u(x)v(x)m_x
\end{equation}
for any $u,v\in A^0(G)$. Let $d^*:A^1(G)\to A^0(G)$ be the adjoint operator of $d:A^0(G)\to A^1(G)$ with respect to the two natural inner products defined on $A^0(G)$ and $A^1(G)$ respectively. Then, by direct computation, it is not hard to see that
\begin{equation}
\Delta=-d^*d.
\end{equation}
This is the same as the smooth case. From this point of view, it is clear that
\begin{equation}\label{eq-integration-by-part}
\vv<\Delta u,v>=-\vv<du,dv>
\end{equation}
for any $u,v\in \R^V$. This implies that $-\Delta$ is a nonnegative self-adjoint operator on $\R^V$. Let
\begin{equation}
0=\mu_1<\mu_2\leq \cdots \leq \mu_{|V|}
\end{equation}
be the eigenvalues of $-\Delta$ on $(G,m,w)$. It is clear that $\mu_1=0$ because constant functions are the corresponding eigenfunctions. Moreover, $\mu_2>0$ because we always assume that $G$ is connected.

We next recall the notion of Bakry-\'Emery curvature on graphs. A weighted graph $(G,m,w)$ is said to  satisfy the Bakry-\'Emery curvature-dimension condition $\CD(K,n)$ if for any $f\in \R^V$ and $x\in V$,
\begin{equation}\label{eq-CD}
\Gamma_2(f,f)(x)\geq \frac{1}{n}(\Delta f)^2(x)+K\Gamma(f,f)(x).
\end{equation}
Here $K$ is a real number and $n$ is positive and can be taken to be $\infty$. Moreover,
\begin{equation}\label{eq-Gamma-def}
\begin{split}
\Gamma(u,v):=\frac12\left(\Delta (uv)-v\Delta u-u\Delta v\right)
\end{split}
\end{equation}
and
\begin{equation}\label{eq-Gamma2}
\Gamma_2(u,v):=\frac{1}2\left(\Delta\Gamma(u,v)-\Gamma(\Delta u,v)-\Gamma(u,\Delta v)\right)
\end{equation}
for any $u,v
\in \R^V$. By direct computation,
\begin{equation}\label{eq-Gamma}
\begin{split}
\Gamma(u,v)(x)=&\frac1{2m_x}\sum_{y\in V}(u(x)-u(y))(v(x)-v(y))w_{xy}\\
=&\frac{1}{2m_x}\vv<du,dv>_{E_x}
\end{split}
\end{equation}
Here $E_x$ means the set of edges in $G$ adjacent to $x$ and for a set $S$ of edges in $G$,
\begin{equation}
\vv<\alpha,\beta>_S:=\sum_{\{x,y\}\in S}\alpha(x,y)\beta(x,y)w_{xy}
\end{equation}
for any $\alpha,\beta\in A^1(G)$.

The following Lichnerowicz-type estimate on graphs was shown independently by several authors \cite{BCLL,KKRT}.
\begin{thm}\label{thm-Lich}
Let $(G,m,w)$ be a connected weighted finite graph satisfying the Bakry-\'Emery curvature-dimension condition $\CD(K,n)$ with $K>0$ and $n>1$. Then
\begin{equation}\label{eq-Lichnerowicz}
  \mu_2(G)\geq \frac{nK}{n-1}.
\end{equation}
\end{thm}
Some discussions on rigidity of \eqref{eq-Lichnerowicz} in the case $n=\infty$ can be found in \cite{LMP2}.

Finally, we need to recall some notions on Steklov eigenvalues for graphs that was introduced in \cite{HHW} and \cite{Pe} recently.

First recall the notion of graphs with boundary.  A pair $(G,B)$ is called a graph with boundary if $G$ is graph and $B\subset V(G)$ such that (i) any two vertices in $B$ are not adjacent, (ii) any vertex in $B$ is adjacent to some vertex
in $\Omega:=V\setminus B$. The set $B$ is called the vertex-boundary of $(G,B)$ and the set $\Omega$ is called the vertex-interior of $(G,B)$. An edge joining a boundary vertex and an interior vertex is called a boundary edge. Note that the notion of graphs with boundary was naturally introduced by Friedman in \cite{FR} when considering nodal domains of graphs.

Let $(G,m,w,B)$ be a weighted graph with boundary. For each $x\in \Omega$, we call
\begin{equation}
\Deg_b(x):=\frac{1}{m_x}\sum_{y\in B}w_{xy}
\end{equation}
the weighted boundary degree at $x$. For any $u\in \R^V$ and $x\in B$, define the normal derivative of $u$ at $x$ as:
\begin{equation}\label{eq-normal-derivative}
\frac{\p u}{\p n}(x):=\frac{1}{m_x}\sum_{y\in V}(u(x)-u(y))\mu_{xy}=-\Delta u(x).
\end{equation}
The reason to define this is that one has the following Green's formula which is a straight forward consequence of \eqref{eq-integration-by-part}:
\begin{equation}\label{eq-Green}
\vv<\Delta u,v>_\Omega=-\vv<du,dv>+\vv<\frac{\p u}{\p n},v>_B.
\end{equation}
Here, for any set $S\subset V$,
\begin{equation}
\vv<u,v>_S:=\sum_{x\in S}u(x)v(x)m_x.
\end{equation}
Because of \eqref{eq-Green}, similar machinery in defining Dirichlet-to-Neumann maps and Steklov eigenvalues in the smooth case can also run in the discrete case. For each $f\in \R^B$, let $u_f$ be the harmonic extension of $f$ into $\Omega$:
\begin{equation}
\left\{\begin{array}{ll}\Delta u_f(x)=0&x\in\Omega\\
u_f(x)=f(x)&x\in B.
\end{array}\right.
\end{equation}
Define the Dirichlet-to-Neumann map $\Lambda:\R^B\to \R^B$ as
\begin{equation}
\Lambda(f)=\frac{\p u_f}{\p n}.
\end{equation}
By \eqref{eq-Green}, it is clear that
\begin{equation}
\vv<\Lambda(f),g>_B=\vv<du_f,du_g>
\end{equation}
for any $f,g\in \R^B$. This implies that $\Lambda$ is a nonnegative self-adjoint operator on $\R^B$. Let
\begin{equation}
0=\sigma_1<\sigma_2\leq\cdots\leq \sigma_{|B|}
\end{equation}
be the eigenvalues of $\Lambda$. It is clear that $\sigma_1=0$ because constant functions are the corresponding eigenfunctions. Moreover, $\sigma_2>0$ because we always assume that $G$ is connected.

We are now ready to state the first main result of this paper, a Lichnerowicz-type estimate for Steklov eigenvalues on graphs.
\begin{thm}\label{thm-Lich-Steklov}
Let $(G,m,w,B)$ be a connected weighted finite graph with boundary. Suppose that $(G,m,w)$ satisfy the Bakry-\'Emery curvature-dimension condition $\CD(K,n)$ with $K>0$ and $n>1$. Then
\begin{equation}\label{eq-Lich-Steklov}
\sigma_2\geq \frac{nK}{n-1}.
\end{equation}
\end{thm}

The Lichnerowicz-type estimate for Stekolov eigenvalues looks quite different with that of the smooth case (see \cite{XX}). One reason for this may be that the Bakry-\'Emery curvature-dimension condition $\CD(K,n)$ on $(G,m,w)$ is a curvature condition on the whole graph including curvature restrictions on the boundary. Another reason may be that, although  Dirichlet-to-Neumann maps on graphs are defined by the same  machinery with the smooth case, the normal derivative
\begin{equation}\label{eq-normal-der-Lap}
\frac{\p u}{\p n}=-(\Delta u)|_B
\end{equation}
is different with the smooth case where the Dirichlet-to-Neumann map is somehow like the square root of the Laplacian operator.

In this paper, we will present a direct proof to Theorem \ref{thm-Lich-Steklov}. We would like to mention that there is another simple proof by observing that
\begin{equation}
\sigma_i\geq \mu_i
\end{equation}
for $i=1,2,\cdots,|B|$ and combining Theorem \ref{thm-Lich} (see \cite{SY2} for details).

The other main results of this paper are considering  the rigidity of \eqref{eq-Lich-Steklov}. For graphs with unit weight, we have the following conclusion for rigidity of \eqref{eq-Lich-Steklov}.
\begin{thm}\label{thm-rigidity-unit}
Let $(G,B)$ be a connected finite graph with boundary, equipped with the unit weight and satisfying the Bakry-\'Emery curvature-dimension condition $\CD(K,n)$ with $K>0$ and $n>1$. Moreover, suppose that $\sigma_2=\frac{nK}{n-1}$. Then,
$|V(G)|=3$ or $4$. When $|V(G)|=3$, one has $K=\frac12$ and $n=2$. Moreover $(G,B)$ is a path with 3 vertices and with the two end points the boundary vertices. When $|V(G)|=4$, one has $n=\infty$ and $K=2$. Moreover $(G, B)$ is one of the following graphs:
\begin{enumerate}
\item $V(G)=\{1,2,3,4\}, E(G)=\{\{1,2\},\{2,3\},\{3,4\},\{4,1\}\}$ and $B=\{1,3\}$;
\item $V(G)=\{1,2,3,4\},E(G)=\{\{1,2\},\{2,3\},\{3,4\},\{4,1\},\{2,4\}\}$ and $B=\{1,3\}$.
\end{enumerate}
\end{thm}

For general weighted graphs,  by trying to squeeze the curvature-dimension condition into the interior of the graph, we have the following rigidity of \eqref{eq-Lich-Steklov}.
\begin{thm}\label{thm-general}
Let $(G,m,w,B)$ be a connected weighted finite graph with boundary. Suppose that $(G,m,w)$ satisfy the Barky-\'Emery dimension-curvature condition $\CD(n,K)$ with $K>0$ and $n>1$. Then $\sigma_2=\frac{nK}{n-1}$ if and only if all following statements are true:
\begin{enumerate}
\item $|B|=2$ and every interior vertex is adjacent to both of the two boundary vertices.
\item let $B=\{1,2\}$. Then $m_1=m_2:=m$ and for $x\in\Omega$, $w_{1x}=w_{2x}:=w_x$;
\item $\Deg(1)=\Deg(2)=\frac{\sum_{x\in\Omega}w_x}{m}=\frac{nK}{n-1}$;
\item For any $x\in \Omega$, $\Deg_b(x)=\frac{2w_x}{m_x}=\frac{(n+2)K}{n-1}$;
\item Either $n=2$ and $|\Omega|=1$, or $n>2$ and  for any $x\in\Omega$ and $f\in\R^\Omega$ with $f(x)=0$,
\begin{equation}\label{eq-rigidity-n}
\begin{split}
&\Gamma^\Omega_2(f,f)(x)-\frac{1}{n-2}(\Delta_\Omega f)^2(x)+\frac{3K}{n-1}\Gamma^\Omega(f,f)(x)+\frac{(n+2)^2K^2}{8m(n-1)^2}\vv<f,f>_\Omega\\
&-\frac{(n+2)K}{(n-1)(n-2)m}\vv<f,1>_\Omega\Delta_\Omega f(x)-\frac{n(n+2)^2K^2}{8(n-2)(n-1)^2m^2}\vv<f,1>_\Omega^2\geq0.
\end{split}
\end{equation}
In particular, when $n=\infty$, the last inequality becomes
\begin{equation}\label{eq-rigidity-infty}
\begin{split}
\Gamma^\Omega_2(f,f)(x)+\frac{K^2}{8m}\vv<f,f>_\Omega
-\frac{K^2}{8m^2}\vv<f,1>_\Omega^2\geq0
\end{split}
\end{equation}
for any $f\in \R^\Omega$ with $f(x)=0$.
\end{enumerate}
Here, $\Gamma_2^\Omega$, $\Gamma^\Omega$ and $\Delta_\Omega$ are the corresponding $\Gamma_2$, $\Gamma$ and $\Delta$ operators for the induced graph on $\Omega$.
\end{thm}
By Theorem \ref{thm-general}, one can construct many weighted graphs satisfying the Bakry-\'Emery curvature-dimension condition $\CD(K,n)$ with $K>0$ and $n>1$, on which the equality of \eqref{eq-Lich-Steklov} holds. For example, first fix a complete graph $\Omega$ equipped with the vertex measure $m_\Omega$ and edge weight $w_\Omega$ such that $(\Omega, m_\Omega, w_\Omega)$ satisfies the Bakry-\'Emery curvature-dimension condition $\CD(n-2,K_\Omega)$ for some $K_\Omega>0$. Then add $B=\{1,2\}$ to $\Omega$ by joining every vertex in $B$ to every vertex in $\Omega$. Arrange the weights of every boundary edge and the measure of the boundary vertices so that (2),(3),(4) in Theorem \ref{thm-general} holds. Note that by multiplying a constant $\lambda$ to $w_\Omega$, the curvature lower bound of $\Omega$ will become $\lambda K_\Omega$. So
\begin{equation}
\Gamma^\Omega_2(f,f)(x)-\frac{1}{n-2}(\Delta_\Omega f)^2(x)\geq \lambda K_\Omega \Gamma^\Omega(f,f)(x)
\end{equation}
on $(\Omega, m_\Omega,\lambda w_\Omega)$. Moreover, note that the  negative terms in \eqref{eq-rigidity-n} are  bounded from above by  some multiples of $\Gamma^\Omega(f,f)(x)$ with the multipliers independent of $\lambda$ when $\lambda$ is large. So, when $\lambda$ is large enough, then the graph equipped with the re-scale weight will satisfy \eqref{eq-rigidity-n}, and hence equality of \eqref{eq-Lich-Steklov} holds on the graph that satisfies the Bakry-\'Emery curvature condition $\CD(K,n)$ with $K>0$ and $n>2$ by Theorem \ref{thm-general}.

Because \eqref{eq-rigidity-n} looks complicated, we first consider rigidity of \eqref{eq-Lich-Steklov} for the case that the induced graph on $\Omega$ is trivial. In this case, $$\Gamma_2^\Omega(f,f)=\Gamma^\Omega(f,f)=\Delta_\Omega f=0.$$ So, \eqref{eq-rigidity-n} becomes very simple and we are able to get a classification in this case.
\begin{thm}\label{thm-rigidity-partial}
Let $(G,m,w,B)$ be a connected weighted finite graph with boundary and $E(\Omega,\Omega)=\emptyset$. Suppose that $(G,m,w)$ satisfy the Bakry-\'Emery curvature-dimension condition $\CD(K,n)$ with $K>0$ and $n>1$, and $\sigma_2=\frac{nK}{n-1}$. Then, $|V(G)|=3$ or $4$. When $|V(G)|=3$, one has $n\geq 2$ and $(G,B)$ is a path on three vertices with the two end points as the boundary vertices. Moreover, If denote the two boundary vertices as $1$ and $2$, and the interior vertex as $x$, then $m_1=m_2:=m$, $w_{1x}=w_{2x}=\frac{mnK}{n-1}$, and $m_x=\frac{2n m}{n+2}$.
When $|V(G)|=4$, one has $n=\infty$ and $(G,B)$ is a square with two diagonal vertices as the boundary vertices. Moreover,  if denote the two boundary vertices as $1$ and $2$, and the interior vertices as $x$ and $y$, then $m_1=m_2=m_x=m_y:=m$ and $w_{1x}=w_{2x}=w_{1y}=w_{2y}=\frac{mK}{2}$.
\end{thm}
By applying the last theorem, one has the following rigidity of \eqref{eq-Lich-Steklov} for graphs with normalized weight.
\begin{thm}\label{thm-rigidity-normal}
Let $(G,m,w,B)$ be a connected finite graph with boundary equipped with a normalized weight. Suppose that $(G,m,w)$ satisfy the curvature-dimension condition $\CD(K,n)$ with $K>0$ and $n>1$ and $\sigma_2=\frac{nK}{n-1}$. Then,
$n=\infty$, $K=1$, and  $V(G)=3$ or $4$. Moreover, $(G,m,w,B)$ is the graphs listed in Theorem \ref{thm-rigidity-partial} with $n=\infty$ and $K=1$.
\end{thm}
For general weighted graphs, we have the following information about the structures of graphs satisfying the Bakry-\'Emery curvature-dimension condition $\CD(K,n)$ with $K>0$ and $n>1$ and that the equality of \eqref{eq-Lich-Steklov} holds derived from \eqref{eq-rigidity-n} and \eqref{eq-rigidity-infty}.

\begin{thm}\label{thm-connect}
Let $(G,m,w,B)$ be a connected weighted finite graph with boundary. Suppose that $(G,m,w)$ satisfy the Barky-\'Emery dimension-curvature condition $\CD(n,K)$ with $K>0$ and $n>1$ and $\sigma_2=\frac{nK}{n-1}$. Then,

(1) When $2<n<\infty$, the induced graph on $\Omega$ can not contained two disjoint balls of radius $2$. In particular, the induced graph on $\Omega$ must be connected and with diameter not greater than $4$.

(2) When $n=\infty$, the induced graph on $\Omega$ can not contained two disjoint balls of radius $2$ unless $(G,B)$ is a square with two diagonal vertices the boundary vertices (i.e.  the case with $|V(G)|=4$ in Theorem \ref{thm-rigidity-partial}.). In particular, if $G$ is not a square, then the induced graph on $\Omega$ must be connected and with diameter not greater than $4$.
\end{thm}

At the end of this section, we would like to mention that, although Dirichlet-to-Neumann maps and Steklov eigenvalues on graphs were introduced very recently by Hua-Huang-Wang \cite{HHW} and Hassannezhad-Miclo \cite{HM} independently, there are already a number of works on the topic. See for example \cite{HHW2,HH,Pe,Pe2,SY,SY2}.

The rest of the paper is organized as follows. In section 2, we prove Theorem \ref{thm-Lich-Steklov} and obtained crucial properties for graphs on which equality of \eqref{eq-Lich-Steklov} holds. Moreover, as a simple application of the properties we obtained, we prove Theorem \ref{thm-rigidity-unit} which gives a classification of graphs with unit weight such that equality of \eqref{eq-Lich-Steklov} holds. In Section 3, we deal with the rigidity of \eqref{eq-Lich-Steklov} for general weighted graphs.
\section{Lichnerowicz estimate and rigidity for unit weighted}
In this section, we prove the main results of the paper. Throughout this section, $u\in \R^V$ is used to denote the harmonic extension of an eigenfunction of $\sigma_2$ for $(G,m,w,B)$. More precisely, $u$ is a function on $V$ satisfying:
\begin{equation}\label{eq-u}
\left\{\begin{array}{ll}\Delta u(x)=0& x\in \Omega\\
\frac{\p u}{\p n}(x)=\sigma_2u(x)&x\in B\\
\vv<u,u>_B=1.
\end{array}\right.
\end{equation}

We first prove Theorem \ref{thm-Lich-Steklov}. The argument is rather classical by just integrating \eqref{eq-CD}.
\begin{proof}[Proof of Theorem \ref{thm-Lich-Steklov}]
 By \eqref{eq-integration-by-part}, \eqref{eq-Gamma}, \eqref{eq-normal-der-Lap} and \eqref{eq-u},
\begin{equation}
\begin{split}
&\sum_{x\in V(G)}\Gamma_2(u,u)(x)m_x\\
=&-\sum_{x\in V}\Gamma(\Delta u,u)(x)m_x\\
=&-\vv<d\Delta u,du>\\
=&\vv<\Delta u,\Delta u>\\
=&\vv<\frac{\p u}{\p n},\frac{\p u}{\p n}>_B\\
=&\sigma_2^2\vv<u,u>_B\\
=&\sigma_2^2.
\end{split}
\end{equation}
Furthermore, by \eqref{eq-Gamma},\eqref{eq-normal-der-Lap} and \eqref{eq-u},
\begin{equation}
\begin{split}
&\sum_{x\in V}\left(\frac{1}{n}(\Delta u(x))^2+K\Gamma(u,u)(x)\right)m_x\\
=&\frac{\sigma_2^2}{n}\vv<u,u>_B+K\vv<du,du>\\
=&\frac{\sigma_2^2}{n}-K\vv<\Delta u,u>\\
=&\frac{\sigma_2^2}{n}+K\vv<\frac{\p u}{\p n},u>_B\\
=&\frac{\sigma_2^2}{n}+K\sigma_2.\\
\end{split}
\end{equation}
Finally, because $(G,m,\mu)$ satisfies the Bakry-\'Emery curvature-dimension condition $\CD(K,n)$, one has
\begin{equation}
\sigma_2^2\geq \frac{\sigma_2^2}{n}+K\sigma_2
\end{equation}
 which gives us the conclusion since $\sigma_2>0$.
\end{proof}

We next  discuss the rigidity of \eqref{eq-Lich-Steklov}. We first have the following lemma mainly claiming that $u$ is also an eigenfunction for $\mu_2$ when the equality in \eqref{eq-Lich-Steklov} holds.
\begin{lem}\label{lem-u-1}
Let the notations and assumptions be the same as in Theorem \ref{thm-Lich-Steklov}. Moreover, suppose that $\sigma_2=\frac{nK}{n-1}$.   Then,
 $\mu_2(G)=\frac{nK}{n-1}$ and $u$ is also an eigenfunction for $\mu_2(G)$. Moreover, $u(x)=0$ for $x\in \Omega$.
\end{lem}
\begin{proof}
 Let $v=u+c$ be such that $\vv<v,1>=0$ where $c$ is some constant. Then,
 \begin{equation}
 \vv<du,du>=\vv<dv,dv>.
 \end{equation}
 Moreover, note that $\vv<u,1>_B=0$, so
 \begin{equation}
 \vv<v,v>_B=\vv<u+c,u+c>_B=\vv<u,u>_B+\vv<c,c>_B\geq \vv<u,u>_B.
 \end{equation}
Then, by the above inequality and Theorem \ref{thm-Lich},
\begin{equation}
\frac{nK}{n-1}=\sigma_2=\frac{\vv<du,du>}{\vv<u,u>_B}\geq\frac{\vv<dv,dv>}{\vv<v,v>_B}\geq\frac{\vv<dv,dv>}{\vv<v,v>}\geq\mu_2\geq \frac{nK}{n-1}.
\end{equation}
So, the inequalities in the expression above must be all equalities which gives us the conclusions of the lemma.
\end{proof}
The following lemma is a key ingredient in the study of rigidity for \eqref{eq-Lich-Steklov}.
\begin{lem}\label{lem-u-2}
Let the notations and assumptions be the same as in Theorem \ref{thm-Lich-Steklov}. Moreover, suppose that $\sigma_2=\frac{nK}{n-1}$.  Then,
\begin{enumerate}
\item for any $x,z\in B$ with $d(x,z)=2$, one has $u(x)=-u(z)$;
\item for any $x\in B$, $u(x)\neq 0$;
\item for any $x\in B$ and $y\in\Omega$, $x\sim y$;
\item $|B|=2$;
\item let $B=\{1,2\}$. Then $m_1=m_2$ and for any $x\in \Omega$, $w_{1x}=w_{2x}$;
\item for each $x\in B$, $\Deg(x)=\frac{nK}{n-1}$;
\item for each $x\in\Omega$, $\Deg_b(x)=\frac{(n+2)K}{n-1}$.
\end{enumerate}
\end{lem}
\begin{proof} By the proof of Theorem \ref{thm-Lich-Steklov}, because the equality of \eqref{eq-Lich-Steklov} holds,
\begin{equation}\label{eq-CD-equal}
\Gamma_2(u,u)(x)=\frac{1}{n}(\Delta u)^2(x)+K\Gamma(u,u)(x)
\end{equation}
for any $x\in V$. By the same argument as in the proof of Theorem 3.4 (b) in \cite[P. 15]{LMP2}, one has
\begin{equation}\label{eq-d=2}
\frac{u(z)+u(x)}{2}=\frac{\sum_{y\in V}u(y)w_{xy}w_{yz}/m_y}{\sum_{y\in V}w_{xy}w_{yz}/m_y}.
\end{equation}
for any $x,z\in V$ with $d(x,z)=2$. We now come to the proofs of each conclusions.

\noindent (1) By Lemma \ref{lem-u-1}, for any $y\in \Omega$, $u(y)=0$. So, by \eqref{eq-d=2},
$$u(z)+u(x)=0.$$\\
(2) Let $N=\{x\in B\ |\ u(x)\neq 0\}$. We first claim that for any $x\in \Omega$, $d(x,N)=1$. Otherwise, let $x\in \Omega$ with $d(x, N)=l\geq 2$. Let $$N\ni x_0\sim x_1\sim x_2\sim\cdots\sim x_l=x$$
be a shortest path joining $x$ to $N$. It is clear that $x_1\in \Omega$. If $x_2\in B$, then by (1), $x_2\in N$ because $d(x_0,x_2)=2$. However, this is impossible since the path above is the shortest path joining $x$ to $N$. So, $x_2\in \Omega$ and moreover $d(x_0,x_2)=2$. This is also because the path we chosen above is the shortest path joining $x$ to $N$. Then, substituting the two vertices $x_0$ and $x_2$ into \eqref{eq-d=2}, we find that the L.H.S. of \eqref{eq-d=2} is nonzero while the R.H.S. is zero which is a contradiction.

The claim above implies that
$$\Gamma(u,u)(y)>0$$
for any $y\in\Omega$. Now, if there is a $x\in B$ with $u(x)=0$. Then,
$$\Delta u(x)=0$$
and
$$\Gamma(u,u)(x)=0,$$
sicne $u|_\Omega=0$ by Lemma \ref{lem-u-1}. Moreover,
$$\Delta\Gamma(u,u)(x)=\frac{1}{m_x}\sum_{y\in\Omega}(\Gamma(u,u)(y)-\Gamma(u,u)(x))\mu_{xy}>0.$$ However, by \eqref{eq-CD-equal},
\begin{equation}
\begin{split}
\frac{1}{2}\Delta\Gamma(u,u)(x)-\Gamma(\Delta u,u)(x)=\Gamma_2(u,u)(x)=0
\end{split}
\end{equation}
which implies that
$$\Delta\Gamma(u,u)(x)=0$$
since
$$\Gamma(\Delta u,u)(x)=\mu_2\Gamma(u,u)(x)=0.$$ This is a contradiction. So, $B=N$.\\
(3) By the proof of the claim in (2), we know that $d(x,B)=1$ for any $x\in \Omega$. Combining this with (1) and (2), we know that each interior vertex will be adjacent exactly to two boundary vertices. Now, if there are $x\in B$ and $y\in \Omega$ such that $d(x,y)=l\geq 2$.  Let
$$x=x_0\sim x_1\sim x_2\sim\cdots\sim x_l=y.$$
be a shortest path joining $x$ and $y$. It is clear that $x_1\in \Omega$. If $x_2\in \Omega$, then $d(x_0,x_2)=2$ because the path we chosen is shortest. Now, similar as in the proof the claim in (2), we get a contradiction by substituting $x_0$ and $x_2$ into \eqref{eq-d=2}. Therefore $x_2\in B$ and $x_3\in \Omega$. Now $d(x_1,x_3)=2$ and they have only one common neighboring  boundary vertex $x_2$. Substituting $x_1$ and $x_3$ into \eqref{eq-d=2}, one can see that the L.H.S of \eqref{eq-d=2} is zero while the R.H.S. is nonzero. This is a contradiction. \\
(4) If there are more than three boundary vertices, then two of them must of the same sign. This violates (1) because of (2) and (3).\\
(5)  Because
$$u(1)=-u(2)\neq 0$$ and
$$u(1)m_1+u(2)m_2=\vv<u,1>_B=0,$$
we know that
$m_1=m_2$. Moreover, because $u(x)=\Delta u(x)=0$ for $x\in \Omega$ by Lemma \ref{lem-u-1},
\begin{equation}
0=\Delta u(x)=\frac{1}{m_x}(u(1)w_{1x}+u(2)w_{2x}).
\end{equation}
So, $w_{x1}=w_{x2}$.\\
(6) Note that, for any $x\in B$,
 $$\frac{nK}{n-1}u(x)=\frac{\p u}{\p n}(x)=\frac1{m_x}\sum_{y\in \Omega}(u(x)-u(y))w_{xy}=\Deg(x)u(x)$$
since $u(y)=0$ for $y\in\Omega$. So, we get the conclusion because $u(x)\neq 0$. \\
(7) Let $B=\{1,2\}$, $m_1=m_2=m$ and $w_{1x}=w_{2x}=w_x$ for any $x\in \Omega$.  Then, by direct computation,
\begin{equation}\label{eq-Gamma-u}
\Gamma(u,u)(x)=\left\{\begin{array}{ll}\frac{\Deg_b(x)}2u(1)^2&x\in\Omega\\
\frac{\Deg(1)}{2}u(1)^2&x=1,2.
\end{array}\right.
\end{equation}
Moreover, by \eqref{eq-CD-equal} and Lemma \ref{lem-u-1},
\begin{equation}\label{eq-Delta-Gamma-u}
\Delta\Gamma(u,u)(x)=-\frac{2K}{n-1}\Gamma(u,u)(x).
\end{equation}
for any $x\in\Omega.$
Moreover, for any $x\in \Omega$,
\begin{equation}\label{eq-Delta-Omega-Gamma-u}
\begin{split}
&\Delta_\Omega \Gamma(u,u)(x)\\
=&\frac1{m_x}\sum_{y\in\Omega}(\Gamma(u,u)(y)-\Gamma(u,u)(x))w_{xy}\\
=&\frac1{m_x}\sum_{y\in V(G)}(\Gamma(u,u)(y)-\Gamma(u,u)(x))w_{xy}-\frac1{m_x}\sum_{y\in B}(\Gamma(u,u)(y)-\Gamma(u,u)(x))w_{xy}\\
=&\Delta\Gamma(u,u)(x)-(\Deg(1)-\Deg_b(x))\Gamma(u,u)(x)\\
=&\left(\Deg_b(x)-\frac{n+2}{n-1}K\right)\Gamma(u,u)(x)
\end{split}
\end{equation}
where we have used \eqref{eq-Gamma-u}, \eqref{eq-Delta-Gamma-u} and (6).

Let $x_0$ be a maximum point of $\Gamma(u,u)(x)$ in $\Omega$ which is equivalent to say that $x_0$ is a maximum point of $\Deg_b(x)$ in $\Omega$ by \eqref{eq-Gamma-u}. Then,
$$\Delta_\Omega \Gamma(u,u)(x_0)\leq 0$$
which implies that
\begin{equation}
\Deg_b(x)\leq\Deg_b(x_0)\leq\frac{n+2}{n-1}K
\end{equation}
for any $x\in \Omega$, by \eqref{eq-Delta-Omega-Gamma-u}. Similarly, by using a minimum point of $\Gamma(u,u)(x)$ in $\Omega$, one can see that
\begin{equation}
\Deg_b(x)\geq \frac{n+2}{n-1}K
\end{equation}
for any $x\in\Omega$. These give us the conclusion.
\end{proof}
We are now ready to prove Theorem \ref{thm-rigidity-unit}, the rigidity of \eqref{eq-Lich-Steklov} for graphs with unit weight.
\begin{proof}[Proof of Theorem \ref{thm-rigidity-unit}] By Lemma \ref{lem-u-2}, we know that, when the graph is of unit weight,
\begin{equation}\label{eq-Degb=2}
2=|B|=\Deg_b(x)=\frac{(n+2)K}{n-1}
\end{equation}
for any $x\in \Omega$, and
\begin{equation}\label{eq-Deg=Omega}
|\Omega|=\Deg(x)=\frac{nK}{n-1}
\end{equation}
for any $x\in B$. So,
\begin{equation}\label{eq-Omega=n}
|\Omega|=\frac{2n}{n+2}\leq 2.
\end{equation}
Therefore, $|\Omega|=1$ or $2$. When $|\Omega|=1$, by \eqref{eq-Degb=2} and \eqref{eq-Omega=n}, $n=2$ and $K=\frac12$. When $|\Omega|=2$, by \eqref{eq-Degb=2} and \eqref{eq-Omega=n} again, $n=\infty$ and $K=2$. The other conclusions of the theorem are not hard to be checked.
\end{proof}

\section{Rigidity for general weighted graphs}
In this section, we discuss rigidity of \eqref{eq-Lich-Steklov} for general weighted graphs. By Lemma \ref{lem-u-2}, we know that for a connected weighted finite graph $(G,m,w,B)$ with boundary satisfying the Bakry-\'Emery curvature-dimension condition $\CD(K,n)$ with $K>0$ and $n>1$, such that the equality of \eqref{eq-Lich-Steklov} holds, the statements (1),(2),(3),(4) in Theorem \ref{thm-general} must be true. So, without further indication,  we assume the graph $(G,m,w,B)$ satisfy (1),(2),(3),(4) in Theorem \ref{thm-general} throughout this section. More precisely, we assume that
\begin{enumerate}
\item[(A1)] $B=\{1,2\}$ and $m_1=m_2=m$;
\item[(A2)] for any $x\in\Omega$, $w_{1x}=w_{2x}=w_x$;
\item[(A3)] $\Deg(1)=\Deg(2)=\frac{\sum_{x\in \Omega}w_x}{m}=\frac{nK}{n-1}:=\Deg$;
\item[(A4)] for any $x\in\Omega$, $\Deg_b(x)=\frac{(n+2)K}{n-1}:=\Deg_b$.
\end{enumerate}
By (A3) and (A4), one has
\begin{equation}
\Deg_b=\frac{2\sum_{x\in\Omega}w_x}{\sum_{x\in \Omega}m_x}=\frac{\sum_{x\in\Omega}w_x}{m}\frac{2m}{V_\Omega}=\frac{2m\Deg}{V_\Omega}.
\end{equation}
So,
\begin{equation}\label{eq-V-Omega}
V_\Omega=\frac{2m\Deg}{\Deg_b}=\frac{2mn}{n+2}.
\end{equation}

We now compute the expressions of some related quantities in terms of some quantities in the interior.

\begin{lem}\label{lem-Gamma} For any $f,g\in \R^V$,
\begin{equation}
\begin{split}
&\Gamma(f,g)(x)\\
=&\Gamma^\Omega(f,g)(x)+\frac{\Deg_b}{2}f(x)g(x)-\frac{\Deg_b}{4}(f(1)+f(2))g(x)-\frac{\Deg_b}{4}(g(1)+g(2))f(x)\\
&+\frac{\Deg_b}{4}(f(1)g(1)+f(2)g(2)),
\end{split}
\end{equation}
for any $x\in\Omega$,
\begin{equation}
\Gamma(f,g)(1)=\frac{\Deg_b}{4m}\vv<f,g>_\Omega-\frac{\Deg_b}{4m}f(1)\vv<g,1>_\Omega-\frac{\Deg_b}{4m}g(1)\vv<f,1>_\Omega+\frac{\Deg}{2}f(1)g(1)
\end{equation}
and similarly,
\begin{equation}
\Gamma(f,g)(2)=\frac{\Deg_b}{4m}\vv<f,g>_\Omega-\frac{\Deg_b}{4m}f(2)\vv<g,1>_\Omega-\frac{\Deg_b}{4m}g(2)\vv<f,1>_\Omega+\frac{\Deg}{2}f(2)g(2).
\end{equation}

\end{lem}
\begin{proof}
The proof is just by direct computation using the assumptions (A1),(A2),(A3),(A4). First, we have
\begin{equation}
  \begin{split}
&\Gamma(f,g)(x)\\
=&\frac{1}{2m_x}\sum_{z\in V}(f(z)-f(x))(g(z)-g(x))w_{xz}\\
=&\frac{1}{2m_x}\sum_{z\in \Omega}(f(z)-f(x))(g(z)-g(x))w_{xz}+\frac{1}{2m_x}\sum_{z\in B}(f(z)-f(x))(g(z)-g(x))w_{xz}\\
=&\Gamma^\Omega(f,g)(x)+\frac{w_x}{2m_x}((f(1)-f(x))(g(1)-g(x))+(f(2)-f(x))(g(2)-g(x)))\\
=&\Gamma^\Omega(f,g)(x)+\frac{\Deg_b}{4}((f(1)-f(x))(g(1)-g(x))+(f(2)-f(x))(g(2)-g(x)))\\
=&\Gamma^\Omega(f,g)(x)+\frac{\Deg_b}{2}f(x)g(x)-\frac{\Deg_b}{4}(f(1)+f(2))g(x)-\frac{\Deg_b}{4}(g(1)+g(2))f(x)\\
&+\frac{\Deg_b}{4}(f(1)g(1)+f(2)g(2)).
  \end{split}
\end{equation}
Moreover,
\begin{equation}
\begin{split}
&\Gamma(f,g)(1)\\
=&\frac{1}{2m}\sum_{x\in\Omega}(f(x)-f(1))(g(x)-g(1))w_x\\
=&\frac{1}{2m}\sum_{x\in\Omega}f(x)g(x)w_x-\frac{1}{2m}f(1)\sum_{x\in\Omega}g(x)w_x-\frac{1}{2m}g(1)\sum_{x\in\Omega}f(x)w_x+\frac{f(1)g(1)}{2m}\sum_{x\in\Omega}w_x\\
=&\frac{1}{2m}\sum_{x\in\Omega}f(x)g(x)m_x\frac{w_x}{m_x}-\frac{1}{2m}f(1)\sum_{x\in\Omega}g(x)m_x\frac{w_x}{m_x}-\frac{1}{2m}g(1)\sum_{x\in\Omega}f(x)m_x\frac{w_x}{m_x}+\frac{\Deg}{2}f(1)g(1)\\
=&\frac{\Deg_b}{4m}\vv<f,g>_\Omega-\frac{\Deg_b}{4m}f(1)\vv<g,1>_\Omega-\frac{\Deg_b}{4m}g(1)\vv<f,1>_\Omega+\frac{\Deg}{2}f(1)g(1)
\end{split}
\end{equation}
\end{proof}
By similar computation as in the proof of the last lemma, we have the following expressions for $\Delta f$.
\begin{lem}\label{lem-Delta}
Let $f\in \R^V$. Then,
\begin{equation}
\Delta f(x)=\Delta_\Omega f(x)-\Deg_b f(x)+\frac{\Deg_b}{2}(f(1)+f(2))
\end{equation}
for any $x\in\Omega$,
\begin{equation}
\Delta f(1)=\frac{\Deg_b}{2m}\vv<f,1>_\Omega-\Deg f(1)
\end{equation}
and similarly,
\begin{equation}
\Delta f(2)=\frac{\Deg_b}{2m}\vv<f,1>_\Omega-\Deg f(2).
\end{equation}
\end{lem}
\begin{proof}
The proof is similar as before by direct computation, we omit it for simplicity.
\end{proof}
Combining the two lemmas above, we have the following expression for $\Gamma_2(f,f).$
\begin{lem}\label{lem-Gamma2}
(1) Let $x\in \Omega$, for any $f\in\R^V$ with $f(x)=0$,
\begin{equation}
\begin{split}
&\Gamma_2(f,f)(x)\\
=&\Gamma^\Omega_2(f,f)(x)+\Deg_b\Gamma^\Omega(f,f)(x)+\frac{\Deg_b^2}{8m}\vv<f,f>_\Omega\\
&+\frac{\Deg_b}{8}\left(3\Deg f^2(1)+2\Deg_b f(1)f(2)+3\Deg f^2(2)\right)-\frac{\Deg_b^2}{4m}\vv<f,1>_\Omega(f(1)+f(2)).
\end{split}
\end{equation}
(2) For any $f\in\R^V$ with $f(1)=0$,
\begin{equation}
\begin{split}
&\Gamma_2(f,f)(1)\\
=&\frac{\Deg_b}{2m}\vv<df,df>_\Omega+\frac{\Deg_b(3\Deg_b-\Deg)}{8m}\vv<f,f>_\Omega+\frac{\Deg_b^2}{8m^2}\vv<f,1>_\Omega^2\\
&+\frac{\Deg_b\Deg}{8}f^2(2)-\frac{\Deg_b^2}{4m}\vv<f,1>_\Omega f(2)
\end{split}
\end{equation}
and similarly, for any $f\in\R^V$ with $f(2)=0$,
\begin{equation}
\begin{split}
&\Gamma_2(f,f)(2)\\
=&\frac{\Deg_b}{2m}\vv<df,df>_\Omega+\frac{\Deg_b(3\Deg_b-\Deg)}{8m}\vv<f,f>_\Omega+\frac{\Deg_b^2}{8m^2}\vv<f,1>_\Omega^2\\
&+\frac{\Deg_b\Deg}{8}f^2(1)-\frac{\Deg_b^2}{4m}\vv<f,1>_\Omega f(1)
\end{split}
\end{equation}
\end{lem}
\begin{proof}
(1) By Lemma \ref{lem-Gamma} and Lemma \ref{lem-Delta},
\begin{equation}\label{eq-Gamma-2-x-1}
\begin{split}
&\Delta\Gamma(f,f)(x)\\
=&\Delta_\Omega\Gamma(f,f)(x)-\Deg_b \Gamma(f,f)(x)+\frac{\Deg_b}{2}(\Gamma(f,f)(1)+\Gamma(f,f)(2))\\
=&\Delta_\Omega\Gamma^\Omega(f,f)(x)+\frac{\Deg_b}{2}\Delta_\Omega f^2(x)-\frac{\Deg_b}{2}(f(1)+f(2))\Delta_\Omega f(x)\\
&-\Deg_b\left(\Gamma^\Omega(f,f)(x)+\frac{\Deg_b}{4}(f^2(1)+f^2(2))\right)\\
&+\frac{\Deg_b}{2}\left(\frac{\Deg_b}{2m}\vv<f,f>_\Omega-\frac{\Deg_b}{2m}(f(1)+f(2))\vv<f,1>_\Omega+\frac{\Deg}{2}(f^2(1)+f^2(2))\right)\\
=&\Delta_\Omega\Gamma^\Omega(f,f)(x)+\frac{\Deg_b^2}{4m}\vv<f,f>_\Omega-\frac14(\Deg_b-\Deg)\Deg_b(f^2(1)+f^2(2))\\
&-\left(\frac{\Deg_b}{2}\Delta_\Omega f(x)+\frac{\Deg_b^2}{4m}\vv<f,1>_\Omega\right)(f(1)+f(2))
\end{split}
\end{equation}
by noting that $$\Gamma^\Omega(f,f)(x)=\frac{1}{2}\Delta_\Omega f^2(x)-f(x)\Delta_\Omega f(x)=\frac{1}{2}\Delta_\Omega f^2(x)$$
since we have assumed that $f(x)=0$.

Moreover, by Lemma \ref{lem-Gamma} and Lemma \ref{lem-Delta} again,
\begin{equation}\label{eq-Gamma-2-x-2}
\begin{split}
&\Gamma(\Delta f,f)(x)\\
=&\Gamma^\Omega(\Delta f,f)(x)-\frac{\Deg_b}{4}(f(1)+f(2))\Delta f(x)+\frac{\Deg_b}{4}(f(1)\Delta f(1)+f(2)\Delta f(2))\\
=&\Gamma^\Omega(\Delta_\Omega f,f)(x)-\Deg_b\Gamma^\Omega(f,f)(x)-\frac{\Deg_b}{4}(f(1)+f(2))\left(\Delta_\Omega f(x)+\frac{\Deg_b}{2}(f(1)+f(2))\right)\\
&+\frac{\Deg_b}{4}\left(\frac{\Deg_b}{2m}\vv<f,1>_\Omega(f(1)+f(2))-\Deg (f^2(1)+f^2(2))\right)\\
=&\Gamma^\Omega(\Delta_\Omega f,f)(x)-\Deg_b\Gamma^\Omega(f,f)(x)\\
&-\frac{\Deg_b}{8}\left((\Deg_b+2\Deg)f^2(1)+2\Deg_bf(1)f(2)+(\Deg_b+2\Deg)f^2(2)\right)\\
&+\left(\frac{\Deg_b^2}{8m}\vv<f,1>_\Omega-\frac{\Deg_b}{4}\Delta_\Omega f(x)\right)(f(1)+f(2)).
\end{split}
\end{equation}
Hence, by \eqref{eq-Gamma-2-x-1} and\eqref{eq-Gamma-2-x-2},
\begin{equation}
\begin{split}
&\Gamma_2(f,f)(x)\\
=&\frac12\Delta\Gamma(\Delta f,f)(x)-\Delta\Gamma(f,f)(x)\\
=&\Gamma^\Omega_2(f,f)(x)+\Deg_b\Gamma(f,f)(x)+\frac{\Deg_b^2}{8m}\vv<f,f>_\Omega\\
&+\frac{\Deg_b}{8}\left(3\Deg f^2(1)+2\Deg_b f(1)f(2)+3\Deg f^2(2)\right)-\frac{\Deg_b^2}{4m}\vv<f,1>_\Omega(f(1)+f(2)).
\end{split}
\end{equation}
(2) By Lemma \ref{lem-Gamma} and Lemma \ref{lem-Delta},
\begin{equation}\label{eq-Gamma-1-1}
\begin{split}
&\Delta\Gamma(f,f)(1)\\
=&\frac{\Deg_b}{2m}\vv<\Gamma(f,f),1>_\Omega-\Deg \Gamma(f,f)(1)\\
=&\frac{\Deg_b}{2m}\left(\vv<\Gamma^\Omega(f,f),1>_\Omega+\frac{\Deg_b}{2}\vv<f,f>_\Omega-\frac{\Deg_b}{2}f(2)\vv<f,1>_\Omega+\frac{\Deg_bV_\Omega}{4}f^2(2)\right)\\
&-\frac{\Deg\Deg_b}{4m}\vv<f,f>_\Omega\\
=&\frac{\Deg_b}{2m}\vv<df,df>_\Omega+\frac{\Deg_b(\Deg_b-\Deg)}{4m}\vv<f,f>_\Omega+\frac{\Deg_b^2V_\Omega}{8m}f^2(2)-\frac{\Deg_b^2}{4m}\vv<f,1>_\Omega f(2),
\end{split}
\end{equation}
where $\vv<du,du>_\Omega:=\vv<du,du>_{E(\Omega,\Omega)}$.
Moreover,
\begin{equation}\label{eq-Gamma-1-2}
\begin{split}
&\Gamma(\Delta f,f)(1)\\
=&\frac{\Deg_b}{4m}\vv<\Delta f,f>_\Omega-\frac{\Deg_b}{4m}\Delta f(1)\vv<f,1>_\Omega\\
=&\frac{\Deg_b}{4m}\vv<\Delta_\Omega f,f>_\Omega-\frac{\Deg_b^2}{4m}\vv<f,f>_\Omega-\frac{\Deg_b^2}{8m^2}\vv<f,1>_\Omega^2+\frac{\Deg_b^2}{8m}\vv<f,1>_\Omega f(2)\\
=&-\frac{\Deg_b}{4m}\vv<d f,df>_\Omega-\frac{\Deg_b^2}{4m}\vv<f,f>_\Omega-\frac{\Deg_b^2}{8m^2}\vv<f,1>_\Omega^2+\frac{\Deg_b^2}{8m}\vv<f,1>_\Omega f(2).
\end{split}
\end{equation}
Combining \eqref{eq-Gamma-1-1} and \eqref{eq-Gamma-1-2},
\begin{equation}
\begin{split}
&\Gamma_2(f,f)(1)\\
=&\frac12\Delta\Gamma(f,f)(1)-\Gamma(\Delta f,f)(1)\\
=&\frac{\Deg_b}{2m}\vv<df,df>_\Omega+\frac{\Deg_b(3\Deg_b-\Deg)}{8m}\vv<f,f>_\Omega+\frac{\Deg_b^2}{8m^2}\vv<f,1>_\Omega^2\\
&+\frac{\Deg_b\Deg}{8}f^2(2)-\frac{\Deg_b^2}{4m}\vv<f,1>_\Omega f(2)
\end{split}
\end{equation}
where we have used \eqref{eq-V-Omega}.
\end{proof}
We are now ready to compute the curvature of a graph satisfying the assumptions (A1),(A2),(A3),(A4).
\begin{thm}\label{thm-curvature}Let $(G,m,w,B)$ be a connected weighted finite graph with boundary satisfying the assumptions (A1),(A2),(A3) and (A4). Then,
\begin{enumerate}
\item $(G,m,w)$ satisfies the Bakry-\'Emery curvature-dimension condition $\CD(K,n)$ at the boundary vertices.
\item $(G,m,w)$ satisfies the Bakry-\'Emery curvature-dimension condition $\CD(K,n)$ at $x\in\Omega$ if and only if either $n=2$ and $|\Omega|=1$, or $n>2$ and
\begin{equation}
\begin{split}
&\Gamma^\Omega_2(f,f)(x)-\frac{1}{n-2}(\Delta_\Omega f)^2(x)+\frac{3K}{n-1}\Gamma^\Omega(f,f)(x)+\frac{(n+2)^2K^2}{8m(n-1)^2}\vv<f,f>_\Omega\\
&-\frac{(n+2)K}{(n-1)(n-2)m}\vv<f,1>_\Omega\Delta_\Omega f(x)-\frac{n(n+2)^2K^2}{8(n-2)(n-1)^2m^2}\vv<f,1>_\Omega^2\geq0
\end{split}
\end{equation}
for any $f\in \R^\Omega$ with $f(x)=0$. In particular, $(G,m,w)$ satisfies the Bakry-\'Emery curvature-dimension condition $\CD(K,\infty)$ at $x\in\Omega$ if and only if
\begin{equation}
\begin{split}
\Gamma^\Omega_2(f,f)(x)+\frac{K^2}{8m}\vv<f,f>_\Omega
-\frac{K^2}{8m^2}\vv<f,1>_\Omega^2\geq0
\end{split}
\end{equation}
for any $f\in \R^\Omega$ with $f(x)=0$.
\end{enumerate}
\end{thm}
\begin{proof}
(1) Note that $(G,m,w)$ satisfies the Bakry-\'Emery curvature-dimension condition $\CD(K,n)$ at $1$ if and only if
\begin{equation}
\Gamma_2(f,f)(1)-\frac{1}{n}(\Delta f)^2(1)-K\Gamma(f,f)(1)\geq 0
\end{equation}
for any $f\in \R^V$ with $f(1)=0$. By Lemma \ref{lem-Gamma}, Lemma \ref{lem-Delta} and Lemma \ref{lem-Gamma2},
\begin{equation}
\begin{split}
&\Gamma_2(f,f)(1)-\frac{1}{n}(\Delta f)^2(1)-K\Gamma(f,f)(1)\\
=&\frac{\Deg_b}{2m}\vv<df,df>_\Omega+\frac{\Deg_b}{4m}\left(\frac{3\Deg_b-\Deg}{2}-K\right)\vv<f,f>_\Omega+\frac{\Deg_b^2}{4m^2}\left(\frac12-\frac1n\right)\vv<f,1>_\Omega^2\\
&+\frac{\Deg_b\Deg}{8}f^2(2)-\frac{\Deg_b^2}{4m}\vv<f,1>_\Omega f(2)\\
=&\frac{\Deg_b}{2m}\vv<df,df>_\Omega+\frac{\Deg_b}{4m}\left(\frac{3\Deg_b-\Deg}{2}-K\right)\vv<f,f>_\Omega+\frac{\Deg_b^2}{4m^2}\left(\frac12-\frac1n\right)\vv<f,1>_\Omega^2\\
&+\frac{\Deg_b\Deg}{8}\left(f(2)-\frac{\Deg_b}{m\Deg}\vv<f,1>_\Omega\right)^2-\frac{\Deg_b^3}{8m^2\Deg}\vv<f,1>_\Omega^2\\
\geq&\frac{\Deg_b}{4m}\left(\frac{3\Deg_b-\Deg}{2}-K\right)\vv<f,f>_\Omega+\frac{\Deg_b^2}{4m^2}\left(\frac12-\frac1n-\frac{\Deg_b}{2\Deg}\right)\vv<f,1>_\Omega^2\\
=&\frac{\Deg_b}{4m}\left(\frac{4K}{n-1}\vv<f,f>_\Omega-\frac{2\Deg_b}{nm}\vv<f,1>_\Omega^2\right)\\
\geq&\frac{\Deg_b}{4m}\left(\frac{4K}{n-1}\vv<f,f>_\Omega-\frac{2\Deg_b V_\Omega}{nm}\vv<f,f>_\Omega\right)\\
=&\frac{\Deg_b}{4m}\left(\frac{4K}{n-1}\vv<f,f>_\Omega-\frac{4\Deg }{n}\vv<f,f>_\Omega\right)\\
=&0
\end{split}
\end{equation}
where we have used the Cauchy-Schwartz inequality and \eqref{eq-V-Omega}. Similarly, the Bakry-\'Emery curvature-dimension condition $\CD(K,n)$ is also satisfied at $2$.

(2) Note that $(G,m,w)$ satisfies the Bakry-\'Emery curvature-dimension condition $\CD(K,n)$ at $x\in\Omega$ if and only if
\begin{equation}
\Gamma_2(f,f)(x)-\frac{1}{n}(\Delta f)^2(x)-K\Gamma(f,f)(x)\geq 0
\end{equation}
for any $f\in \R^V$ with $f(x)=0$. By Lemma \ref{lem-Gamma}, Lemma \ref{lem-Delta} and Lemma \ref{lem-Gamma2},
\begin{equation}\label{eq-Gamma-2}
\begin{split}
&\Gamma_2(f,f)(x)-\frac{1}{n}(\Delta f)^2(x)-K\Gamma(f,f)(x)\\
=&\Gamma^\Omega_2(f,f)(x)-\frac{1}{n}(\Delta_\Omega f)^2(x)+(\Deg_b-K)\Gamma^\Omega(f,f)(x)+\frac{\Deg_b^2}{8m}\vv<f,f>_\Omega\\
&+\frac{\Deg_b}{8}\bigg(\left(3\Deg-\frac{2\Deg_b}n-2K\right) f^2(1)+\left(2-\frac{4}n\right)\Deg_b f(1)f(2)\\
&+\left(3\Deg-\frac{2\Deg_b}n-2K\right) f^2(2)\bigg)-\left(\frac{\Deg_b^2}{4m}\vv<f,1>_\Omega+\frac{\Deg_b}{n}\Delta_\Omega f(x)\right)(f(1)+f(2))\\
=&\Gamma^\Omega_2(f,f)(x)-\frac{1}{n}(\Delta_\Omega f)^2(x)+\frac{3K}{n-1}\Gamma^\Omega(f,f)(x)+\frac{\Deg_b^2}{8m}\vv<f,f>_\Omega\\
&+\frac{1}{8}\left(1-\frac2n\right)\Deg_b^2(f(1)+f(2))^2-\left(\frac{\Deg_b}{4m}\vv<f,1>_\Omega+\frac{1}{n}\Delta_\Omega f(x)\right)\Deg_b(f(1)+f(2)).
\end{split}
\end{equation}
When $n\in (1,2)$, the expression \eqref{eq-Gamma-2} will be negative if we choose $f\in \R^V$ such that $f|_\Omega\equiv 0$ and $f(1)+f(2)$ is large enough.  So, in this case, the graph will not satisfy the Bakry-\'Emery curvature-dimension $\CD(K,n)$ at $x$. Moreover, when $n=2$ and $|\Omega|\geq 2$, let $f(y)=1$ for any $y\in \Omega\setminus\{x\}$, so that
\begin{equation}
  \frac{\Deg_b}{4m}\vv<f,1>_\Omega+\frac{1}{n}\Delta_\Omega f(x)>0.
\end{equation}
Then, by letting $f(1)+f(2)$ be large enough, the expression \eqref{eq-Gamma-2} will be negative. So, in this case, the graph will not satisfy the Bakry-\'Emery curvature-dimension $\CD(K,n)$ at $x$. Furthermore, it is clear that when $n=2$ and $|\Omega|=1$, the expression \eqref{eq-Gamma-2} is zero. So, in this case, the graph will satisfy the Bakry-\'Emery curvature-dimension $\CD(K,n)$ at $x$.

In the following, we will assume that $n>2$. We can first fix suitable $f(1)$ and $f(2)$ so that
$$\frac{1}{8}\left(1-\frac2n\right)\Deg_b^2(f(1)+f(2))^2-\left(\frac{\Deg_b}{4m}\vv<f,1>_\Omega+\frac{1}{n}\Delta_\Omega f(x)\right)\Deg_b(f(1)+f(2))$$
achieves its minimum
\begin{equation}
-\frac{2n}{n-2}\left(\frac{\Deg_b}{4m}\vv<f,1>_\Omega+\frac{1}{n}\Delta_\Omega f(x)\right)^2.
\end{equation}
So, in this case, the graph satisfies the Bakry-\'Emery curvature-dimension condition $\CD(n,K)$ at $x$ if and only if
\begin{equation}
\begin{split}
&\Gamma^\Omega_2(f,f)(x)-\frac{1}{n}(\Delta_\Omega f)^2(x)+\frac{3K}{n-1}\Gamma^\Omega(f,f)(x)+\frac{\Deg_b^2}{8m}\vv<f,f>_\Omega\\
&-\frac{2n}{n-2}\left(\frac{\Deg_b}{4m}\vv<f,1>_\Omega+\frac{1}{n}\Delta_\Omega f(x)\right)^2\geq 0
\end{split}
\end{equation}
for any $f\in \R^\Omega$ with $f(x)=0$. Simplified the last expression, we arrive at
\begin{equation}
\begin{split}
&\Gamma^\Omega_2(f,f)(x)-\frac{1}{n-2}(\Delta_\Omega f)^2(x)+\frac{3K}{n-1}\Gamma^\Omega(f,f)(x)+\frac{(n+2)^2K^2}{8m(n-1)^2}\vv<f,f>_\Omega\\
&-\frac{(n+2)K}{(n-1)(n-2)m}\vv<f,1>_\Omega\Delta_\Omega f(x)-\frac{n(n+2)^2K^2}{8(n-2)(n-1)^2m^2}\vv<f,1>_\Omega^2\geq0
\end{split}
\end{equation}
for any $f\in \R^\Omega$ with $f(x)=0$. This completes the proof of the Theorem.
\end{proof}
We are now ready to prove Theorem \ref{thm-general}. In fact, it is a direct consequence of Lemma \ref{lem-u-2} and Theorem \ref{thm-curvature}.
\begin{proof}[Proof of Theorem \ref{thm-general}]. If $(G,m,w,B)$ satisfies the Bakry-\'Emery curvature-dimension condition $\CD(K,n)$ with $K>0$ and $n>1$, and the equality of \eqref{eq-Lich-Steklov} holds. Then, by Lemma \ref{lem-u-2}, we know that (1),(2),(3),(4) must be true. Moreover, by Theorem \ref{thm-curvature}, we know that (5) must be true.

Conversely, it clear that for a graph satisfying (1),(2),(3) and (4), one has $\sigma_2=\frac{n}{n-1}K$. Moreover, by Theorem \ref{thm-curvature}, a graph satisfying (1),(2),(3),(4),(5) must satisfy the Bakry-\'Emery curvature condition $\CD(K,n)$.
\end{proof}

We next come to prove Theorem \ref{thm-rigidity-partial}, a rigidity of \eqref{eq-Lich-Steklov} for general weighted graphs with  trivial induced graph on the interior.
\begin{proof}[Proof of Theorem \ref{thm-rigidity-partial}] When $n=2$, there are nothing need to be proved by Theorem \ref{thm-general}. So, in the following, we assume that $n>2$. If $|\Omega|=1$, then there is also nothing need to be proved by Theorem \ref{thm-general}. So, we further assume that $|\Omega|\geq 2$. By Theorem \ref{thm-general}, we know that equality of \eqref{eq-Lich-Steklov} holds if and only if assumptions (A1),(A2),(A3),(A4) are satisfied and
\begin{equation}\label{eq-f}
\frac{(n+2)^2K^2}{8m(n-1)^2}\vv<f,f>_\Omega-\frac{n(n+2)^2K^2}{8(n-2)(n-1)^2m^2}\vv<f,1>_\Omega^2\geq0
\end{equation}
for any $f\in\R^\Omega$ which is vanished at some vertex in $\Omega$. For each $x\in \Omega$, let $f=\chi_{\Omega\setminus\{x\}}$ in  \eqref{eq-f}. We get
\begin{equation}
V_{\Omega\setminus\{x\}}\leq \frac{(n-2)m}{n}.
\end{equation}
Then,
\begin{equation}
m_x=V_\Omega-V_{\Omega\setminus\{x\}}\geq \frac{(n^2+4)m}{n(n+2)}\geq \frac{nm}{n+2}=\frac{V_\Omega}{2}
\end{equation}
by using \eqref{eq-V-Omega}. The equality holds if and only if $n=\infty$. This implies that when $n<\infty$, it is impossible for $|\Omega|\geq 2$. Moreover, when $n=\infty$, then $|\Omega|=2$ and $m_x=m$ for any $x\in\Omega$. This completes the proof of the theorem.
\end{proof}
We next come to prove Theorem \ref{thm-rigidity-normal}, the rigidity of \eqref{eq-Lich-Steklov} for graphs with normalized weight.
\begin{proof}[Proof of Theorem \ref{thm-rigidity-normal}]Because the graph is equipped with a normalized weight, by Lemma \ref{lem-u-2},
\begin{equation}
1=\Deg(x)=\frac{n K}{n-1}
\end{equation}
for any $x\in B$, and moreover,
\begin{equation}
1\geq \Deg_b(x)=\frac{(n+2)K}{n-1}\geq \frac{nK}{n-1}=1
\end{equation}
for any $x\in \Omega$. This means that $n=\infty$, $K=1$ and $\Deg_b(x)=1$ for any $x\in\Omega$. Because $\Deg(x)=1$ for any $x\in V$, we know that $E(\Omega,\Omega)=\emptyset$. Then, by Theorem \ref{thm-rigidity-partial}, we complete the proof of the theorem.
\end{proof}

Finally, we come to prove theorem \ref{thm-connect}.
\begin{proof}[Proof of Theorem \ref{thm-connect}]
(1) Let $B^\Omega_x(2)$ and $B^\Omega_y(2)$ be two disjoint balls of radius $2$ in the induced graph on $\Omega$. Let $f=\chi_{\Omega\setminus B^\Omega_{x}(2)}$.  Then, it is clear that
\begin{equation}
\Gamma^\Omega_2(f,f)(x)=\Delta_\Omega f(x)=\Gamma^\Omega(f,f)(x)=0.
\end{equation}
Substituting this into \eqref{eq-rigidity-n}, we get
\begin{equation}
V_{\Omega\setminus B^\Omega_{x}(2)}\leq\frac{(n-2)m}{n}
\end{equation}
which implies that
\begin{equation}
V_{B^\Omega_{x}(2)}=V_\Omega-V_{\Omega\setminus B^\Omega_{x}(2)}=\frac{(n^2+4)m}{n(n+2)}>\frac{V_\Omega}{2}.
\end{equation}
Similarly,
\begin{equation}
V_{B^\Omega_{y}(2)}>\frac{V_\Omega}{2}.
\end{equation}
However, this impossible since the two balls are disjoint. \\
(2) Let $B^\Omega_x(2)$ and $B^\Omega_y(2)$ be two disjoint balls of radius $2$ in the induced graph on $\Omega$. By the same argument in (1), we know that
\begin{equation}
V_{B^\Omega_x(2)}=V_{B^\Omega_y(2)}=\frac{V_\Omega}{2}=m.
\end{equation}
So $\Omega=B^\Omega_x(2)\cup B^\Omega_y(2)$. We are only left to show that $B^\Omega_x(2)=\{x\}$ and $B^\Omega_y(2)=\{y\}$. Otherwise, let $z$ be another vertex in $B^\Omega_x(2)$. Let $f\in\R^\Omega$ be such that
\begin{equation}
f(\xi)=\left\{\begin{array}{ll}0&\xi\in B^\Omega_x(2)\mbox{ and }\xi\neq z\\
1&\xi=z\\
c&\xi\in B_y^\Omega(2).
\end{array}\right.
\end{equation}
Substituting this into \eqref{eq-rigidity-infty}, we have
\begin{equation}
\Gamma_2^\Omega(f,f)(x)+\frac{K^2}{8}\left(\frac{m_z}{m}-\frac{m_z^2}{m^2}\right)-\frac{K^2m_z}{4m}\cdot c\geq0.
\end{equation}
However, this is impossible when $c$ is large enough since $\Gamma_2^\Omega(f,f)$ depends only on the values of $f$ on $B^\Omega_x(2)$. This completes the proof of the conclusion.

\end{proof}

\end{document}